\documentclass[10pt]{amsart}

\usepackage{enumerate}
\usepackage{graphicx}
\usepackage{float}
\usepackage{placeins}
\usepackage{mdframed}
\usepackage{amssymb}
\usepackage{esint}
\usepackage{cool}
\usepackage[all,cmtip]{xy}
\usepackage{mathtools}
\usepackage{amstext} 
\usepackage{array}   
\usepackage[shortlabels]{enumitem}

\newcolumntype{L}{>{$}l<{$}} 
\newtheorem{theorem}{Theorem}[section]
\newtheorem{lemma}[theorem]{Lemma}
\newtheorem{cor}[theorem]{Corollary}
\newtheorem{prop}[theorem]{Proposition}
\newtheorem{setup}[theorem]{Setup}
\theoremstyle{definition}
\newtheorem{definition}[theorem]{Definition}

\newtheorem{observation}[theorem]{Observation}

\newtheorem{chunk}[theorem]{}
\theoremstyle{remark}
\newtheorem{remark}[theorem]{Remark}

\newtheorem{the context}[theorem]{The Context}

\numberwithin{equation}{theorem}
\numberwithin{equation}{section}

\newcommand{\pd}{\operatorname{pd}}

\newcommand{\rank}{\operatorname{rank}}

\newcommand{\soc}{\operatorname{Soc}}

\newcommand{\type}{\operatorname{type}}

\newcommand{\Span}{\operatorname{Span}}

\newcommand{\tor}{\operatorname{Tor}}
\newcommand{\im}{\operatorname{Im}}

\newcommand{\ideal}[1]{\mathfrak{#1}}
\newcommand{\m}{\ideal{m}}

\renewcommand{\geq}{\geqslant}
\renewcommand{\leq}{\leqslant}

\renewcommand{\hom}{\Hom}

\newcommand{\Hom}{\operatorname{Hom}}

\newcommand{\socle}{\operatorname{Soc}}

\newcommand{\maps}[5]{\xymatrix{#1 \ar[r]^-{#3} & #2 \\
#4 \ar@{|->}[r] & #5 \\}}

\newcommand{\mfa}{\mathfrak{a}}
\newcommand{\pf}{\textrm{Pf}}

\def\im{\operatorname{im}}

\begin{document}
\title[Resolutions and Tor Algebra Structures]{Resolution and Tor Algebra Structures of Grade 3 Ideals Defining Compressed Rings}
\author{Keller VandeBogert }
\date{\today}

\maketitle

\begin{abstract}
    Let $R=k[x,y,z]$ be a standard graded $3$-variable polynomial ring, where $k$ denotes any field. We study grade $3$ homogeneous ideals $I \subseteq R$ defining compressed rings with socle $k(-s)^{\ell} \oplus k(-2s+1)$, where $s \geq3$ and $\ell \geq 1$ are integers. The case for $\ell =1$ was previously studied in \cite{vandebogert2019structure}; a generically minimal resolution was constructed for all such ideals. The paper \cite{vandebogert2020trimming} generalizes this resolution in the guise of (iterated) trimming complexes. In this paper, we show that all ideals of the above form are resolved by an iterated trimming complex. Moreover, we apply this machinery to construct ideals $I$ such that $R/I$ is a ring of Tor algebra class $G (r)$ for some fixed $r \geq2$, and $R/I$ may be chosen to have arbitrarily large type. In particular, this provides a new class of counterexamples to a conjecture of Avramov not already constructed by Christensen, Veliche, and Weyman in \cite{christensen2019trimming}.
\end{abstract}

\section{Introduction}

Let $(R, \m , k)$ be a regular local ring with maximal ideal $\m$ and residue field $k$. A result of Buchsbaum and Eisenbud (see \cite{buchsbaum1977algebra}) established that any quotient $R/I$ of $R$ with projective dimension $3$ admits the structure of an associative commutative differental graded (DG) algebra. Later, a complete classification of the multiplicative structure of the Tor algebra $\tor_\bullet^R (R/I , k)$ for such quotients was established by Weyman in \cite{weyman1989structure} and Avramov, Kustin, and Miller in \cite{avramov1988poincare}.

Given an $\m$-primary ideal $I = (\phi_1 , \dots , \phi_n) \subseteq R$, one can ``trim" the ideal $I$ by, for instance, forming the ideal $(\phi_1 , \dots , \phi_{n-1} ) + \m \phi_n$. This process is used by Christensen, Veliche, and Weyman (see \cite{christensen2019trimming}) in the case that $R/I$ is a Gorenstein ring to produce ideals defining rings with certain Tor algebra classification, negatively answering a question of Avramov in \cite{avramov2012cohomological}. 

This trimming procedure also arises in classifying certain type $2$ ideals defining compressed rings. More precisely, it is shown in \cite{vandebogert2019structure} that every homogeneous grade $3$ ideal $I \subseteq k[x,y,z]$ defining a compressed ring with socle $\soc (R/I) = k(-s) \oplus k(-2s+1)$ is obtained by trimming a Gorenstein ideal. A complex is produced that resolves all such ideals; it is generically minimal. This resolution is then used to bound the minimal number of generators and, consequently, parameters arising in the Tor algebra classification.

In the current paper, we continue with this theme. In particular, much of the work done in \cite{vandebogert2019structure} can be generalized to the case that $I$ is a homogeneous grade $3$ ideal $I \subseteq k[x,y,z]$ defining a compressed ring with $\soc (R/I) = k(-s)^\ell \oplus k(-2s+1)$ for some $\ell \geq 1$. The values $s$ and $2s-1$ are interesting because they are extremal; more precisely, it is not possible to have a quotient ring as above with socle minimally generated in degrees $s_1$ and $s_2$, with $s_2 \geq 2s_1$ (see \cite[Proposition 1.10]{christensen2021type2} for a proof of this fact). We show that all such ideals are then obtained as \emph{iterated} trimmings of a Gorenstein ideal (see Proposition \ref{genset}), and the Tor algebra structure may be computed in similar fashion.

We employ a piece of machinery from \cite{vandebogert2020trimming}, namely an \emph{iterated trimming complex}, in order to resolve all homogeneous grade $3$ ideals $I$ with $\soc (R/I) = k(-s)^\ell \oplus k(-2s+1)$ for some $\ell \geq 1$. This complex arises as a natural generalization of the complex constructed in \cite{vandebogert2019structure} and has applications to the resolution of a variety of other classes of ideals, explored in \cite{vandebogert2020trimming}.

The paper is organized as follows. In Section \ref{ittrimcx}, we recall the construction of trimming complexes as given in \cite{vandebogert2020trimming}. In Sections \ref{hightype} and \ref{toralgst}, we consider grade $3$ homogeneous ideals $I \subseteq k[x,y,z]$ (a standard graded polynomial ring over a field $k$) defining a compressed ring with $\soc (R/I ) = k(-s)^\ell \oplus k(-2s+1)$. We first show that all such ideals define a ring with tipping point (see Definition \ref{tippt}) $s$ and type $\leq s+2$. Using this information we deduce the previously mentioned fact that $I$ is obtained as the iterated trimming of a grade $3$ Gorenstein ideal. In particular, $I$ is resolved by the complex of Section \ref{ittrimcx}. Moreover, we show precisely when such rings have Tor algebra class $G(r)$, and $r$ is computed in terms of the minimal number of generators of $I$ and the variable $\ell$. 

In Section \ref{sec:realizability}, we show that there are ideals of arbitrarily large type of class $G(r)$, for any $r \geq 2$. The construction of these ideals is remarkably simple and is a generalization of the ideals constructed in \cite{vandebogert2019structure}. One sees that the machinery of iterated trimming complexes allows for a quick proof that these ideals satisfy all of the required hypotheses of Proposition \ref{classg}. In particular, these ideals provide a class of counterexamples to the conjecture of Avramov that is not already contained in \cite{christensen2019trimming}.

\section{Iterated Trimming Complexes}\label{ittrimcx}

In this section, we recall the construction of trimming complexes. All proofs of the following results may be found in Section $2$ and $3$ of \cite{vandebogert2020trimming}; the purpose here is to give a concise and efficient introduction to the machinery, as it will be used in later sections.

\begin{setup}\label{setup4}
Let $R=k[x_1,\dots, x_n]$ be a standard graded polynomial ring over a field $k$. Let $I \subseteq R$ be a homogeneous ideal and $(F_\bullet, d_\bullet)$ denote a homogeneous free resolution of $R/I$. 

Write $F_1 = F_1' \oplus \Big( \bigoplus_{i=1}^m Re_0^i \Big)$, where each $e^i_0$ generates a free direct summand of $F_1$. Using the isomorphism
$$\hom_R (F_2 , F_1 ) = \hom_R (F_2,F_1') \oplus \Big( \bigoplus_{i=1}^m \hom_R (F_2 , Re^i_0) \Big)$$
write $d_2 = d_2' + d_0^1 + \cdots + d^m_0$, where $d_2' \in \hom_R (F_2,F_1')$, $d^i_0 \in \hom_R (F_2 , Re^i_0)$. Let $\mfa_i$ denote any homogeneous ideal with
$$d^i_0 (F_2) \subseteq \mfa_i e^i_0,$$
and $(G^i_\bullet , m^i_\bullet)$ be a homogeneous free resolution of $R/\mfa_i$. 

Use the notation $K' := \im (d_1|_{F_1'} : F_1' \to R)$, $K^i_0 := \im (d_1|_{Re^i_0} : Re^i_0 \to R)$, and let $J := K' + \mfa_1 \cdot K^1_0+ \cdots + \mfa_m \cdot K_0^m$.
\end{setup}

\begin{prop}\label{prop:it1stmap}
Adopt notation and hypotheses of Setup \ref{setup4}. Then for each $i=1,\dots,m$ there exist maps $q^i_1 : F_2 \to G^i_1$ such that the following diagram commutes:
$$\xymatrix{& F_2 \ar[dl]_-{q^i_1} \ar[d]^{{d^i_0}'} \\
G_1 \ar[r]_{m^i_1} & \mfa_i \\},$$
where ${d^i_0}' : F_2 \to R$ is the composition
$$\xymatrix{F_2 \ar[r]^{d^i_0} & Re^i_0 \ar[r] & R\\},$$
the second map sending $e^i_0 \mapsto 1$. 
\end{prop}

\begin{prop}\label{prop:ittheyexist}
Adopt notation and hypotheses as in Setup \ref{setup4}. Then for each $i=1, \dots , m$ there exist maps $q^i_k : F_{k+1} \to G^i_{k}$ for all $k \geq 2$ such that the following diagram commutes:
$$\xymatrix{F_{k+1} \ar[d]_{q^i_k} \ar[r]^-{d_{k+1}} & F_k \ar[d]^{q^i_{k-1}} \\
G^i_k \ar[r]_{m^i_k} & G^i_{k-1} \\}$$
\end{prop}

Theorem \ref{thm:itres} is one of the main results of \cite{vandebogert2020trimming} and, in the current paper, will end up being an explicit free resolution for the ideals considered in later sections.

\begin{theorem}\label{thm:itres}
Adopt notation and hypotheses as in Setup \ref{setup4}. Then the mapping cone of the morphism of complexes
\begin{equation}\label{itcomx}
\xymatrix{\cdots \ar[r]^{d_{k+1}} &  F_{k} \ar[ddd]^{\begin{pmatrix} q_{k-1}^1 \\
\vdots \\
q_{k-1}^m \\
\end{pmatrix}}\ar[r]^{d_{k}} & \cdots \ar[r]^{d_3} & F_2 \ar[rrrr]^{d_2'} \ar[ddd]^{\begin{pmatrix} q_1^1 \\
\vdots \\
q_1^m \\
\end{pmatrix}} &&&& F_1' \ar[ddd]^{d_1} \\
&&&&&&& \\
&&&&&&& \\
\cdots \ar[r]^-{\bigoplus m^i_k} & \bigoplus_{i=1}^m G^i_{k-1} \ar[r]^-{\bigoplus m^i_{k-1}} & \cdots \ar[r]^-{\bigoplus m^i_2} & \bigoplus_{i=1}^m G^i_1 \ar[rrrr]^-{-\sum_{i=1}^\ell m^i_1(-)\cdot d_1(e^i_0)} &&&& R \\}\end{equation}
is acyclic and forms a resolution of the quotient ring defined by $K' + \mfa_1 \cdot K^1_0+ \cdots + \mfa_m \cdot K_0^m$. 
\end{theorem}

\begin{definition}\label{def:ittrimcx}
The \emph{iterated trimming complex} associated to the data of Setup \ref{setup4} is the complex of Theorem \ref{thm:itres}.
\end{definition}

Of course, even if the input data or the iterated trimming complex is minimal, the vertical maps of diagram \ref{itcomx} may have scalar entries. One can still compute the Betti numbers, however:

\begin{cor}\label{cor:ittorrk}
Adopt notation and hypotheses of Setup \ref{setup4}. Assume furthermore that the complexes $F_\bullet$ and $G_\bullet$ are minimal. Then for $i \geq 2$,
$$\dim_k \tor_i^R (R/J , k) = \rank F_i + \sum_{j=1}^m  \rank G^j_i - \rank \Bigg( \begin{pmatrix} q_i^1 \\
\vdots \\
q_i^m \\
\end{pmatrix} \otimes k\Bigg) - \rank \Bigg( \begin{pmatrix} q_{i-1}^1 \\
\vdots \\
q_{i-1}^m \\
\end{pmatrix} \otimes k\Bigg).$$
Similarly,
$$\mu (J) = \mu(K) -m + \sum_{j=1}^m \mu(\mfa_j) -\rank \Bigg( \begin{pmatrix} q_1^1 \\
\vdots \\
q_1^m \\
\end{pmatrix} \otimes k\Bigg) .\qquad \qquad \square $$ 
\end{cor}

\section{Compressed Rings of Higher Type}\label{hightype}

In this section and Section \ref{toralgst} we generalize some of the work done in \cite{vandebogert2019structure}. In particular, Corollary \ref{cor:ittrimresolvesyou} says that the ideals introduced in Setup \ref{setup2} are resolved by the iterated trimming complex of Theorem \ref{thm:itres}. In Proposition \ref{classg}, we show that under suitable hypotheses, all ideals introduced in Setup \ref{setup2} define rings of Tor algebra class $G(r)$ for some $r$ (see Definition \ref{def:classg}). This information will be used to produce rings of Tor algebra class $G$ with arbitrarily large type in Section \ref{sec:realizability}.

\begin{definition}
Let $A$ be a local Artinian $k$-algebra, where $k$ is a field and $\m$ denotes the maximal ideal. The top socle degree is the maximum $s$ with $\m^s \neq 0$ and the socle polynomial of $A$ is the formal polynomial $\sum_{i=0}^s c_i z^i$, where
$$c_i = \dim_k \frac{\socle (A) \cap \m^i}{\socle (A) \cap \m^{i+1}}.$$
An Artinian $k$-algebra is \emph{standard graded} if it is generated as an algebra in degree $1$. 
\end{definition}

\begin{definition}\label{compdef}
A standard graded Artinian $k$-algebra $A$ with embedding dimension $e$, top socle degree $s$, and socle polynomial $\sum_{i=0}^s c_i z^i$ is \emph{compressed} if
$$\dim_k \m^i/\m^{i+1} = \min \Big\{ \binom{e-1+i}{i} , \sum_{\ell=0}^s c_\ell \binom{e-1+\ell-i}{\ell - i} \Big\}$$
for $i =0, \dots , s$.
\end{definition}

\begin{chunk}\label{notation}
Let $n \geq 1$ be an integer and $k$ denote a field of arbitrary characteristic. Let $V$ be a vector space of dimension $n$ over $k$. Give the symmetric algebra $S(V) =: R$ and divided power algebra $D(V^*)$ the standard grading (that is, $S_1(V) = V$, $D_1 (V^*) = V^*$). The notation $S_i := S_i (V)$ denotes the degree $i$ component of the symmetric algebra on $V$. Similarly, the notation $D_i := D_i (V^*)$ denotes the degree $i$ component of the divided power algebra on $V^*$.

Given a homogeneous $I \subseteq S(V)$ defining an Artinian ring, there is an associated inverse system $0:_{D(V^*)} I$. Similarly, for any finitely generated graded submodule $N \subseteq D(V^*)$ there is a corresponding homogeneous ideal $0:_{S(V)} N$ defining an Artinian ring.  

If $I$ is a homogeneous ideal with associated inverse system minimally generated by elements $\phi_1 , \ \dots , \phi_k$ with $\deg \phi_i = s_i$, then there are induced vector space homomorphisms
$$\Phi_i : S_i \to \bigoplus_{j=1}^k D_{s_j - i}$$
sending $f \mapsto (f \cdot \phi_1 , \dots , f \cdot \phi_k)$. 
\end{chunk}

\begin{definition}\label{tippt}
Adopt notation as in \ref{notation}. Let $I \subseteq S(V)$ be a homogeneous ideal with associated inverse system minimally generated by elements $\phi_1 , \ \dots , \phi_k$ with $\deg \phi_i = s_i$. Let $m$ denote the first integer for which $\Phi_m$ is a surjection. Then $m$ is called the \emph{tipping point} of $I$; this is well defined since the rank of the domain and codomain of each $\Phi_i$ is increasing/decreasing in $i$, respectively (and the codomain is eventually $0$).
\end{definition}

Notice that the tipping point $m$ of $I$ and the initial degree of $I$ agree unless $\Phi_m$ is an isomorphism, in which case the initial degree of $I$ is $m+1$ (see Remark $1.8$ of \cite{miller2018free}).

\begin{prop}[\cite{miller2018free}, Lemma 1.13]\label{proplol}
Adopt notation as in \ref{notation}. Let $\phi$ be a homogeneous element of $D(V^*)$ of degree $s$. Then the tipping point of the ideal $0 :_{S(V)} \phi$ is $\lceil s/2 \rceil$. In addition, the induced maps $\Phi_i$ satisfy the following properties for every integer $i$.
\begin{enumerate}[(a)]
    \item $\hom_{k} (\Phi_i , k) = \Phi_{s-i}$
    \item $\Phi_i$ is surjective if and only if $\Phi_{s-i}$ is injective.
\end{enumerate}
\end{prop}

\begin{setup}\label{setup2}
Let $k$ be a field and let $R = k[x,y,z]$ be a standard graded polynomial ring over a field $k$. Let $I \subset R$ be a grade $3$ homogeneous ideal defining a compressed ring with $\soc (R/I) = k(-s)^\ell \oplus k(-2s+1)$, where $s \geq 3$. 

Write $I = I_1 \cap I_2 \cap \cdots \cap I_\ell \cap I_t$ for $I_1,\dots, I_\ell$ homogeneous grade $3$ Gorenstein ideals defining rings with socle degrees $s$ and $I_t$ a homogeneous grade $3$ Gorenstein ideal defining a ring with socle degree $2s-1$. The notation $R_+$ will denote the irrelevant ideal ($R_{>0}$).
\end{setup}

\begin{theorem}\label{tippingpt}
Adopt notation and hypotheses of Setup \ref{setup2}. Then the tipping point of $I$ is equal to $s$. In particular, $\ell \leq s+1$. 
\end{theorem}

The proof of Theorem \ref{tippingpt} will follow after a series of lemmas that will give tight upper and lower bounds on the tipping point, forcing equality.

\begin{lemma}\label{lower}
Adopt notation and hypotheses of Setup \ref{setup2}. Then the tipping point of $I$ is $\geq s$.
\end{lemma}

\begin{proof}
Adopt the notation of \ref{notation}; we may view $R$ as $S(V)$ for some $3$-dimensional vector space over $k$. By counting initial degrees, we eliminate all possibilities except for the case that $\Phi_{s-1} : S_{s-1} \to D_1^{\oplus \ell} \oplus D_s$ is an isomorphism and $I$ has initial degree $s$. Counting ranks, this implies $\ell s = 1-s \leq 0$, which is a clear contradiction.
\end{proof}

\begin{lemma}\label{iscomp}
Adopt notation and hypotheses of Setup \ref{setup2}. Then $I_t$ defines a compressed ring.
\end{lemma}

\begin{proof}
Adopt the notation of \ref{notation}; we may view $R$ as $S(V)$ for some $3$-dimensional vector space over $k$. Let $\phi_i \in D_s$ denote the inverse system for each $I_i$ and $\phi_t \in D_{2s-1}$ denote the inverse system for $I_t$. By Lemma \ref{lower}, the maps $\Phi_i$ for $i \geq s$ are surjective; Proposition \ref{proplol} guarantees that the map $f \mapsto f \cdot \phi_t$ is surjective for $f \in S_s$. 

For $i>s$, the maps $\Phi_i : S_i \to D_{2s-1-i}$ are identically the maps $f \mapsto f \cdot \phi_t$ for $f \in S_i$. By assumption, these are surjections; by Proposition \ref{proplol}, $I_t$ defines a compressed ring. 
\end{proof}

\begin{lemma}\label{upper}
Adopt notation and hypotheses of Setup \ref{setup2}. Then the tipping point of $I$ is $\leq s+1$.
\end{lemma}

\begin{proof}
Suppose for sake of contradiction that the tipping point is $\geq s+2$. Adopt the notation of \ref{notation}; we may view $R$ as $S(V)$ for some $3$-dimensional vector space over $k$. Let $\phi_i \in D_s$ denote the inverse system for each $I_i$ and $\phi_t \in D_{2s-1}$ denote the inverse system for $I_t$. 

By Proposition \ref{proplol}, $I_t$ has tipping point $s$. If $I$ has tipping point $\geq s+2$, then $\Phi_{s+1} : S_{s+1} \to D_{s-2}$ is injective; this is impossible by counting ranks.
\end{proof}

\begin{lemma}\label{gortype}
Adopt notation and hypotheses of Setup \ref{setup2}. Let $\phi_1 , \dots , \phi_{s+1}, \psi_1 , \dots , \psi_b$ denote a minimal generating set for $I_t$, where $\deg \phi_i = s$, $\deg \psi_i = s+1$. Then the ideal
$$(\phi_1 , \dots , \phi_{s+1-\ell} , \psi_1 , \dots , \psi_b)+ R_+\phi_{s+2-\ell} + \cdots + R_+ \phi_{s+1}$$
defines a ring of type $\ell+1$. In particular, $(I_t)_{\geq s+1}$ defines a ring of type $s+2$.
\end{lemma}

\begin{remark}
By Proposition $3.3$ of \cite{vandebogert2019structure} combined with Lemma \ref{iscomp}, $I_t$ is minimally generated by homogeneous forms $\phi_1 , \dots , \phi_{s+1}$, $\psi_1 , \dots , \psi_b$, where $\deg \phi_i = s$, $\deg \psi_i = s+1$ and $b <s+1$, so it makes sense to choose a generating set as in the statement of Lemma \ref{gortype}.
\end{remark}

\begin{proof}
This is a consequence of Corollary \ref{cor:ittorrk}. Let $J:=(\phi_1 , \dots , \phi_{s+1-\ell} , \psi_1 , \dots , \psi_b)+ R_+\phi_{s+2-\ell} + \cdots + R_+ \phi_{s+1}$. In the notation of Theorem \ref{thm:itres}, notice that $G_\bullet^j = K_\bullet$, the Koszul complex resolving $R_+$, for all $j=1,\dots , i$. Counting degrees on the diagram of Proposition \ref{prop:ittheyexist}, we find
$$\deg q_2^j \geq s-1 > 0,$$
so that $q_2^j \otimes k = 0$, for all $j=1, \dots , i$. Since $q_3^j = 0$ for each $j=1, \dots , i$, Corollary \ref{cor:ittorrk} implies that
\begin{equation*}
    \begin{split}
        \dim_k \tor_3^R (R/J, k) &= \rank F_3 + \sum_{j=1}^i \rank K_3 \\
        &= i+1. \\
    \end{split}
\end{equation*}
\end{proof}

\begin{proof}[Proof of Theorem \ref{tippingpt}]
By Lemmas \ref{upper} and \ref{lower}, we only need to check that the tipping point cannot equal $s+1$. Assume that $I$ has tipping point $=s+1$. This implies that the map
$$\Phi_s : S_s \to D_0^{\oplus \ell} \oplus D_{s-1}$$
is an injection. Counting ranks, $\ell+\binom{s+1}{2} \geq \binom{s+2}{2}$. If there is equality, then $\Phi_s$ is an isomorphism, implying that the tipping point is $\leq s$. Thus there is strict inequality, and $\ell \geq s+2$.

This implies $I$ has type $\geq s+3$ and initial degree $s+1$. However, Lemma \ref{iscomp} forces $I_t$ to be compressed. Counting ranks in each homogeneous component and using the definition of compressed, we must have $I = (I_t)_{\geq s+1}$. By Lemma \ref{gortype}, $I$ has type $s+2$; this contradiction yields the result.
\end{proof}

\begin{cor}\label{intcomp}
Adopt notation and hypotheses of Setup \ref{setup2}. Then for each $i \leq \ell$, the ideal
$$I_i \cap \cdots \cap I_\ell \cap I_t$$
defines a compressed ring with socle $k(-s)^{\ell - i} \oplus k(-2s+1)$. 
\end{cor}

\begin{proof}
Adopt the notation of \ref{notation}; we may view $R$ as $S(V)$ for some $3$-dimensional vector space over $k$. Let $\phi_i \in D_s$ denote the inverse system for each $I_i$ and $\phi_t \in D_{2s-1}$ denote the inverse system for $I_t$.

For $j<s$, the map $f \mapsto f \cdot \phi_t$ is injective, since $I_t$ defines a compressed ring by Lemma \ref{iscomp}. Similarly, for $j \geq s$, the map $\Phi_j : S_j \to D_j^{\oplus \ell} \oplus D_{2s-1-j}$ associated to the ideal $I$ is surjective. The map $\Phi_j'$ associated to the ideal $I_i \cap \cdots \cap I_\ell \cap I_t$ is the composition of $\Phi_j$ with the canonical projection $D_j^{\oplus \ell} \oplus D_{2s-1-j} \to D_j^{\oplus \ell-i} \oplus D_{2s-1-j}$. As a composition of surjections, $\Phi_j'$ is a surjection for $j \geq s$. 
\end{proof}

The following Proposition provides us with a generating set for ideals as in Corollary \ref{intcomp} that is surprisingly simple. In other words, it says that any ideal defining a compressed ring with socle $k(-s)^\ell \oplus k (-2s+1)$ is obtained as an \emph{iterated} trimming of a Gorenstein ideal. Corollary \ref{cor:ittrimresolvesyou} is an immediate consequence of this observation.

\begin{prop}\label{genset}
Adopt notation and hypotheses of Setup \ref{setup2}. Then there exists a minimal generating set 
$$\phi_1 , \dots , \phi_{s+1}, \ \psi_1 , \dots , \psi_b$$
for $I_t$ such that
$$I=(\phi_1 , \dots , \phi_{s+1-\ell} , \psi_1 , \dots , \psi_b)+ R_+\phi_{s+2-\ell} + \cdots + R_+ \phi_{s+1},$$
where $\deg \phi_i = s$, $\deg \psi_i = s+1$, and $b<s+1$.
\end{prop}

\begin{proof}
Observe that, by definition of compressed,
$$\dim_k (I)_s = s+1- \ell.$$
Choose a basis $\phi_1 , \dots , \phi_{s+1-\ell}$ for $I_s$; notice that $\dim_k (I_t)_s = s+1$, so we may extend this set to a basis
$$\phi_1 , \dots , \phi_{s+1}$$
for $(I_t)_s$. Since $I_{\geq s+1} = (I_t)_{\geq s+1}$, there exist elements $\psi_1 , \dots , \psi_b \in (I_t)_{s+1}$ such that
$$I_{s+1} = (R_+(\phi_1 , \dots , \phi_{s+1}))_{s+1} + \Span_k \{ \psi_1 , \dots , \psi_b \}.$$
In particular, the assumption that $I$ defines a compressed ring forces every minimal generating set to be concentrated in two consecutive degrees. This immediately yields that
$$I=(\phi_1 , \dots , \phi_{s+1-\ell} , \psi_1 , \dots , \psi_b)+ R_+\phi_{s+2-\ell} + \cdots + R_+ \phi_{s+1}.$$
\end{proof}

\begin{cor}\label{cor:ittrimresolvesyou}
Adopt notation and hypotheses of Setup \ref{setup2}. Then the complex of Theorem \ref{thm:itres} is a free resolution of $R/I$. Moreover, if $s$ is even and $I_t$ is generic, this resolution is minimal.
\end{cor}

\begin{proof}
By Proposition $3.4$ of \cite{vandebogert2019structure}, a generic choice of $I_t$ for $s$ even will have quadratic minimal presenting matrix. In the notation of Theorem \ref{thm:itres}, each $G^i_\bullet$ is the Koszul complex resolving $k$, and $F_\bullet$ is the minimal free resolution of $I_t$. Counting degrees on the diagrams of Propositions \ref{prop:it1stmap} and \ref{prop:ittheyexist}, each $q_k^i$ has entries generated in positive degree for $1 \leq i \leq \ell$, whence $q_k^i \otimes k = 0$. 
\end{proof}

\section{Tor Algebra Structures for Higher Type Ideals}\label{toralgst}

In this section, we generalize the results of \cite{vandebogert2019structure} to the case of higher type. Proposition \ref{classg} is the main result of this section and will be used to construct interesting examples of rings with Tor algebra class $G(r)$ in Section \ref{sec:realizability}, but most of the work done for this result is contained in the proof of Theorem \ref{toralg}. To begin the section, we recall the definition of the Tor algebra class $G(r)$:

\begin{definition}[\cite{avramov1988poincare}, Theorem $2.1$]\label{def:classg}
Let $(R,\m,k)$ be a regular local ring with $I\subset \m^2$ and ideal such that $\pd_R (R/I) = 3$. Let $T_\bullet := \tor_\bullet^R (R/I , k)$. Then $R/I$ has Tor algebra class $G(r)$ if, for $m = \mu(I)$ and $t = \type (R/I)$, there exist bases for $T_1$, $T_2$, and $T_3$
$$e_1 , \dots , e_m, \quad f_1 , \dots , f_{m+t-1} , \quad g_1 , \dots , g_t,$$
respectively, such that the only nonzero products are given by
$$e_i f_i = g_1 = f_i e_i, \quad 1 \leq i \leq r.$$
Such a Tor algebra structure has
$$T_1 \cdot T_1 = 0, \quad \rank_k (T_1 \cdot T_2 ) = 1, \quad \rank_k (T_2 \to \hom_k (T_1 , T_3) ) = r,$$
where $r \geq 2$.
\end{definition}

The proof of Theorem \ref{toralg} is essentially that of \cite[Theorem 2.4]{christensen2019trimming} but in iterated form.
\begin{theorem}\label{toralg}
Adopt the notation and hypotheses of Setup \ref{setup2}. Then the rank of the induced map
$$\delta_I : \tor_2^R (R/I , k) \to \hom_k (\tor_1^R (R/I , k) , \tor_3^R (R/I , k) )$$
is at least $\mu(I) - 3\ell$. 
\end{theorem}

\begin{proof}
Throughout the proof, use the notation
$$T^{A}_i := \tor_i^R (A,k),$$
where $A$ is any $R$-module. Notice that by Proposition \ref{genset}, 
$$I_t/I \cong k^{\ell}.$$
Considering the long exact sequence of $\tor$ associated to the short exact sequence
$$\xymatrix{0 \ar[r] & \frac{I_t}{I} \ar[r]^-{\iota} & \frac{R}{I} \ar[r]^-{p} & \frac{R}{I_t} \ar[r] & 0 \\},$$
one counts ranks to find that
$$\rank \tor_2^R (p , k) = \mu (I) - 2\ell, \quad \rank \hom_k (\tor_1^R (p,k ) , T_3^{R/I_t} ) = \mu(I_t) - \ell.$$
Consider the following commutative diagram:
$$\xymatrix{ T_2^{R/I_t} \ar[r]^-{\delta_{I_t}} & \hom_k (T_1^{R/I_t} , T_3^{R/I_t} ) \ar[d]^-{\hom_k (\tor_1^R (p,k ) , T_3^{R/I_t} )}\\
T_2^{R/I} \ar@{=}[d] \ar[u]^-{\tor_2 (p , k)} \ar[r]^-{\varepsilon} & \hom_k (T_1^{R/I} , T_3^{R/I} ) \\
 T_2^{R/I} \ar[r]^-{\delta_I} &  \hom_k (T_1^{R/I} , T_3^{R/I} ) \ar[u]_-{\hom_k (T_1^{R/I} ,\tor_3^R (p,k) )} \\}. $$
 Since $I_t$ is Gorenstein, $\delta_{I_t}$ is an isomorphism. This implies that $\rank \delta_I \geq \rank \varepsilon \geq (\mu(I) - 2 \ell ) - \ell$, which yields the result. 
\end{proof}

\begin{cor}\label{leqclassg}
Adopt the notation and hypotheses of Setup \ref{setup2} and let $b$ be as in Lemma \ref{gortype}. If $\ell \leq s+b-1-\min \{ \ell b , 3 \}$, then $R/I$ has Tor algebra class $G(r)$ for some $r \geq \mu(I) - 3 \ell$. 
\end{cor}

\begin{proof}
It suffices to show that, in the notation of the proof of Theorem \ref{toralg}, the map $\delta_I$ has rank at least $2$. By Corollary \ref{cor:ittorrk},
$$\mu(I) = \mu(I_t) + 2 \ell -\rank \Bigg( \begin{pmatrix} q_1^1 \\
\vdots \\
q_1^m \\
\end{pmatrix} \otimes k\Bigg).$$
Since $\mu(I_t) = s+1+b$ and $\rank \Bigg( \begin{pmatrix} q_1^1 \\
\vdots \\
q_1^m \\
\end{pmatrix} \otimes k\Bigg) \leq \min \{ \ell b , 3 \}$, we deduce that
$$\rank (\delta_I) \geq s+1+b-\min\{ \ell b , 3 \}-\ell \geq 2.$$
\end{proof}

The following is a generalization of Lemma $8.6$ in \cite{vandebogert2019structure}; in simpler words, it shows that iteratively trimming a Gorenstein ideal tends to preserve the Tor-algebra class while changing the homological parameters. 

\begin{prop}\label{classg}
Adopt the notation and hypotheses of Setup \ref{setup2} and let $b$ be as in Lemma \ref{gortype}. If $\ell \leq s+b-1-\min \{ \ell b , 3 \}$, then $R/I$ has Tor algebra class $G( \mu (I) - 3 \ell)$. 
\end{prop}

\begin{proof}
In the notation of the proof of Theorem \ref{toralg}, it suffices to show that
$$\rank \delta_I \leq \mu(I) - 3 \ell.$$
Observe that $\dim_k (T_1^{R/I})_s = s+1-\ell$, so that
$$\dim_k (T_1^{R/I})_{s+1} = \mu(I) - s - 1 + \ell.$$
Moreover, counting ranks on the degree $s+1$ homogeneous strand of the long exact sequence of Tor associated to the short exact sequence
$$0 \to \frac{I_t}{I} \to \frac{R}{I} \to \frac{R}{I_t} \to 0,$$
we obtain
$$\dim_k (T_1^{R/I})_{s+1}  - \dim_k (T_2^{R/I})_{s+1} = 3 \ell$$
$$\implies \dim_k (T_2^{R/I})_{s+1} = \mu(I) - s-1-2\ell.$$
Similarly, a rank count on the degree $s+2$ strand yields
$$\dim_k (T_2^{R/I} )_{s+2} = s+4\ell.$$
By counting degrees, the only nontrivial products can occur between $(T_1^{R/I})_{s}, \ (T_2^{R/I})_{s+2}$ and $(T_1^{R/I})_{s+1}, \ (T_2^{R/I})_{s+1}$; this implies that
\begingroup\allowdisplaybreaks
\begin{align*}
    \rank \delta_I &\leq \dim_k (T_1^{R/I})_{s} + \dim_k (T_2^{R/I})_{s+1} \\
    &= s+1-\ell + \mu(I) - s - 1 -2 \ell \\
    &= \mu(I) - 3 \ell. \\
\end{align*}
\endgroup
Combining this with Corollary \ref{leqclassg}, we find that $R/I$ must be class $G(\mu(I) - 3 \ell)$. 
\end{proof}

\section{Realizability in The Higher Type Case}\label{sec:realizability}

In this section, we construct construct ideals with Tor algebra class $G(r)$ and arbitrarily large type. To begin with, we recall a collection of matrices introduced in Section $7$ of \cite{vandebogert2019structure} (which were inspired by matrices considered in \cite{christensen2019trimming}). It will be particularly easy to apply the construction of Theorem \ref{thm:itres} to the submaximal pfaffians of these matrices to produce interesting classes of ideals attaining Tor algebra class $G(r)$ but with arbitrarily large type. This is stated more precisely in Theorem \ref{thm:hightyperealizability}, which is one of the main results of the paper.

\begin{definition}
Let $U_m^{j}$ (for $j \leq m$) denote the $m\times m$ matrix with entries from the polynomial ring $R=k[x,y,z]$ defined by:
$$(U^{j}_m)_{i,m-i} = x^2, \quad (U^{j}_m)_{i,m-i+1} = z^2, \quad (U^{j}_m)_{i,m-i+2} = y^2 \ \textrm{for} \  i \leq m-j$$
$$(U^{j}_m)_{i,m-i} = x, \quad (U^{j}_m)_{i,m-i+1} = z, \quad (U^{j}_m)_{i,m-i+2} = y \ \textrm{for} \  i >m-j$$
and all other entries are defined to be $0$.
\end{definition}

To see the pattern, we have:
$$U_2^1 = \begin{pmatrix}
       x^{2}&z^{2}\\
       z&y\end{pmatrix}, \ U_3^1 = \begin{pmatrix}
       0&x^{2}&z^{2}\\
       x^{2}&z^{2}&y^{2}\\
       z&y&0\end{pmatrix}, \ U_3^2 = \begin{pmatrix}
       0&x^{2}&z^{2}\\
       x&z&y\\
       z&y&0\end{pmatrix}$$
\begin{definition}\label{def:Vmatrices}
Define $V_m^{j}$ (for $j< m$) to be the $(2m+1)\times (2m+1)$ skew symmetric matrix
$$V^{j}_m := \begin{pmatrix}
O & O_{x^2} & (U_m^{j})^T \\
-(O_{x^2})^T & 0 & ^{y^2}O \\
-U_m^{j} & -(^{y^2}O)^T & O \\
\end{pmatrix}$$
and if $j=m$, then $V_m^m$ is the skew symmetric matrix
$$V^{j}_m := \begin{pmatrix}
O & O_{x^2} & (U_m^{m})^T \\
-(O_{x^2})^T & 0 & ^{y}O \\
-U_m^{m} & -(^{y}O)^T & O \\
\end{pmatrix}$$
\end{definition}

\begin{observation}\label{obs:Vmatbtable}
Let $V_m^j$ be as in Definition \ref{def:Vmatrices}. Then the ideal of submaximal pfaffians $\textrm{Pf} (V_m^j)$ is a grade $3$ Gorenstein ideal with graded Betti table
$$\begin{tabular}{L|L L L L}
     & 0 & 1 & 2 & 3  \\
     \hline 
   0  & 1 & 0 & 0 & 0 \\
    
   2m-j-1 & 0 & 2m+1-j & j & 0 \\
    
   2m-j & 0 & j & 2m+1-j & 0 \\
    
   4m-2j-1 & 0 & 0 & 0 &1 \\
\end{tabular}$$
In particular, $\textrm{Pf} (V_m^j)$ defines a compressed ring.
\end{observation}

With the matrices of Definition \ref{def:Vmatrices} in hand, we are ready to prove our main result:

\begin{theorem}\label{thm:hightyperealizability}
Let $R=k[x,y,z]$ be a standard graded polynomials ring, where $k$ is a field. Let $r \geq 2$ and $N \geq 1$ be integers with $r+N \geq 5$. Then there exists a homogeneous grade $3$ ideal $I$ with $\type (R/I) \geq N$ and such that $R/I$ has Tor algebra class $G(r)$.   
\end{theorem}

\begin{proof}
Observe that by Proposition \ref{classg}, it suffices to produce a grade $3$ homogeneous ideal $I$ defining a compressed ring of type $\ell \geq N$ with $\mu (I) - 3 \ell = r$.

Assume first that $r+N$ is even. Set $m:= \frac{r+N-2}{2}$ and let $M:= V^0_m$. The ideal of submaximal Pfaffians $\pf (M)$ has $r+N-1$ minimal generators. Consider the ideal
$$I:= (\pf_1 (M) , \dotsc , \pf_r (M)) + R_+ (\pf_{r+1} (M), \dotsc , \pf_{r+N-1} (M)).$$
Let $(F_\bullet , d_\bullet)$ denote the minimal free resolution of $\pf (M)$ and $(G_\bullet , m_\bullet)$ the Koszul complex resolving $R_+$. In the notation of Theorem \ref{thm:itres}, the maps $q_1^i : F_2 \to G_1$ satisfy $q_1^i \otimes k = 0$ for all $i=1 , \dotsc , N-1$. This follows by a simple degree count, using the fact that $d_2$ has only quadratic entries. Employing Corollary \ref{cor:ittorrk}, $I$ is minimally generated by $r+3(N-1)$ elements. 

Similarly, since $d_3$ has entries in degree $>2$, a degree count shows $q_2^i \otimes k = 0$ for all $i=1, \dotsc N-1$. This means that $R/I$ is a ring of type $N$. Using Corollary \ref{cor:ittorrk} to count Tor ranks in each graded component, $R/I$ has Betti table
$$\begin{tabular}{L|L L L L}
     & 0 & 1 & 2 & 3  \\
     \hline 
   0  & 1 & 0 & 0 & 0 \\
    
   r+N-3 & 0 & r & 0 & 0 \\
    
   r+N-2 & 0 & 3(N-1) & r+4(N-1) & N-1 \\
    
   2(r+N-2)-1 & 0 & 0 & 0 &1 \\
\end{tabular}$$
In particular, $R/I$ is compressed with $\socle (R/I) \cong k(-(r+N-2))^{N-1} \oplus k(-2(r+N-2)+1)$. Since $r+N-2 \geq 3$, $I$ falls into the hypotheses of Proposition \ref{classg} and hence has Tor algebra class $G(r+3(N-1)-3(N-1)) = G(r)$.

Assume now that $r+N$ is odd. Set $m:= \frac{r+N-1}{2}$ and let $M:= V^0_m$. The ideal of submaximal Pfaffians $\pf (M)$ has $r+N$ minimal generators. Consider the ideal
$$I:= (\pf_1 (M) , \dotsc , \pf_r (M)) + R_+ (\pf_{r+1} (M), \dotsc , \pf_{r+N} (M)).$$
By an argument identical to the even case, one finds that $I$ is minimally generated by $r+3N$ elements and defines a compressed ring of type $N+1$. Since $r+N -1 \geq 3$, $I$ falls into the hypotheses of Proposition \ref{classg} and hence has Tor algebra class $G(r+3N - 3N) = G(r)$. 
\end{proof}

\bibliographystyle{amsplain}
\bibliography{biblio}
\addcontentsline{toc}{section}{Bibliography}

\end{document}